\documentclass[12pt]{amsart}

\usepackage{amsmath, amsthm, amssymb}
\usepackage{url} 
\usepackage{graphicx}

\usepackage{xcolor}
\usepackage{moreverb}

\usepackage{fancyvrb}

\def\r{\mathbb R}

\def\grad{\textrm{grad}}

\usepackage{listings}
 \lstnewenvironment{code}[1][]
  {\lstset{#1}}
  {}

\newtheorem{theorem}{Theorem}[section]

\theoremstyle{definition}
\newtheorem{definition}[theorem]{Definition}

\newtheorem{remark}[theorem]{Remark}
\newtheorem{lemma}[theorem]{Lemma}

\begin{document}

\markboth{}
{Separable hypersurfaces with constant sectional
curvature}

%
 %

\title[Separable hypersurfaces with constant sectional
curvature]{Classification of separable hypersurfaces with constant sectional
curvature}

\author{Muhittin Evren Aydin}

\address{Department of Mathematics, Faculty of Science, Firat University, Elazig,  23200 Turkey}
\email{meaydin@firat.edu.tr }

\author{Rafael López}

\address{Departamento de Geometr\'{\i}a y Topolog\'{\i}a,  Universidad de Granada, 18071 Granada, Spain}
\email{ rcamino@ugr.es}

\author{Gabriel-Eduard Vîlcu}

\address{Department of Mathematics and Informatics, Bucharest Polytechnic University, Faculty of Applied Sciences, Splaiul Independenţei 313, 060042 Bucharest, Romania}
\address{“Gheorghe Mihoc-Caius Iacob” Institute of Mathematical Statistics and Applied Mathematics of the Romanian Academy, 050711 Bucharest, Romania}
\address{Research Center in Geometry, Topology and Algebra, Faculty of Mathematics and Computer Science, University of Bucharest, Academiei 14, Bucharest 010014, Romania}
\email{gabriel.vilcu@upb.ro}

\keywords{separable hypersurface, translation hypersurface, homothetical hypersurface, sectional curvature, Gaussian curvature}
\subjclass{53A07}
\begin{abstract}
In this paper, we give a full classification of the separable
hypersurfaces of constant sectional curvature in the Euclidean $n$-space
$\r^n$. In  dimension $n=3$, this classification was   solved by Hasanis and L\'opez [Manuscripta Math. 166, 403-417 (2021)].
When $n>3$,  we
prove that the separable hypersurfaces of null sectional curvature are
three particular families of such hypersurfaces. Finally,  we prove that
hyperspheres are the only separable hypersurfaces with nonzero constant
sectional curvature.
\end{abstract}
\maketitle

\section{Introduction}

\quad The study of submanifolds of constant sectional curvature in different ambient spaces is one of the important topics in submanifold theory, originating in the investigation of surfaces with constant Gaussian curvature in the Euclidean 3-space (for more details see \cite{GT2018}). When the ambient is a space form, explicit examples of such submanifolds can be constructed using a very useful tool developed in \cite{DT2012} under the name of \emph{Ribaucour transformation}. Moreover, in this setting of an ambient space form, the geometry of hypersurfaces having constant sectional curvature is well understood (see \cite{Li2023}). On the other hand, if the ambient is not a space form, obtaining classification results for submanifolds of constant sectional curvature with arbitrary codimensions is a very challenging question. However, even in this context, some classification results can be obtained using a technique known as \emph{Tsinghua principle}, originally
discovered by Li, Vrancken and Wang  at Tsinghua University in 2013. Using this interesting principle, it has been demonstrated in \cite{Yao2022} the nonexistence of locally conformally flat real hypersurfaces in the complex quadric with dimension $\geq 3$.
We would like to emphasize that many interesting  results concerning the sectional curvature of hypersurfaces in various ambient spaces were obtained in the last decades (see, e.g. \cite{Anti2021,Cheng2022,CI90,Kim94,Man2011,Sack60}).
Notice that the latest result was established in \cite{GK2023}, where the authors obtained the complete classification of hypersurfaces of constant isotropic curvature in a real space form, under some topological assumptions: completeness, connectedness, orientability. Recall that any manifold having constant sectional curvature has  the isotropic curvature also constant, while the converse statement is true in dimension $\geq 5$.

\quad In this paper, we are going to investigate the separable hypersurfaces with constant
sectional curvature in the Euclidean $n$-dimensional space $\r^n$. Recall that separable hypersurfaces are a generalization of some well-known families of hypersurfaces, such as translation hypersurfaces and homothetical hypersurfaces.
In the Euclidean 3-space, a translation surface can be constructed by translating a planar curve over another planar curve. This notion was generalized to a translation hypersurface as the sum of $n$ planar curves. More precisely, a hypersurface $M^{n-1}$ of $\r^n$ is said to be a \emph{translation hypersurface} if $M^{n-1}$ is the graph of a function $F$ defined by
\begin{equation}\label{sum}
  F(x_1,...,x_{n-1})=\sum_{i=1}^{n-1}f_i(x_i),
\end{equation}
where $x_1,...,x_{n-1}$ are the Cartesian coordinates of $\r^{n-1}$ and $f_i$ is a smooth real valued function of single variable, $i=1,...,n-1$ (see \cite{Dil91}).
Note that various results on the geometry of translation surfaces and hypersurfaces were stated in \cite{Chen2003,HL2021,Lima2014,Liu99,s}. For more generalizations of translation hypersurfaces and recent progress, see \cite{Lima2019,MM1,MM2,Ruiz21,Y2019}.

\quad On the other hand, a hypersurface $M^{n-1}$ of $\r^n$ is said to be \emph{homothetical} or \emph{factorable} if  $M^{n-1}$ is the graph of a function $F$ defined by
\begin{equation}\label{prod}
F(x_1,...,x_{n-1})=\prod_{i=1}^{n-1}f_i(x_i),
\end{equation}
where again $x_1,...,x_{n-1}$ denote  the canonical coordinates of  $\r^{n-1}$, while
$f_1,…,f_{n-1}$ are smooth real valued functions of single variable \cite{Jiu2007}. Several curvature properties of homothetical surfaces and hypersurfaces were established in \cite{Jiu2007,LM2013,VW95}.

\quad 
Separable surfaces, which include translation and factorable surfaces as particular subfamilies, have been investigated in geometry since the 19th century (see \cite{Weing}). In particular, separable surfaces with zero mean curvature are known from the works of Scherk, Weingarten, Schwarz and Fr\'echet among others (see \cite[Chapter II, Section 5]{n}). It is clear that separable surfaces
form a class of surfaces with its own interest, which can be extended to separable hypersurfaces as follows.

\begin{definition} \cite{cww}
A hypersurface $M^{n-1}$ of the Euclidean $n$-space $\r^n$ ($n \geq 3$) is
said to be {\it separable}
if it can be expressed as
$$
M^{n-1}=\{ (x_1,...,x_n) \in \r^n:\sum_{i=1}^{n}f_i(x_i)=0 \},
$$
where $f_1,...,f_n$ are smooth real valued functions of single variable defined in certain intervals $I_1,...,I_n$ of $\r$, with $\sum_{i=1}^{n}f'_i(x_i)\neq 0$, for every $x_i\in I_i$, $i=1,...,n$.
\end{definition}

\quad It is clear that any translation hypersurface is a particular type of separable hypersurface. Moreover,
taking logarithms in the equation of a homothetical hypersurface, namely
$
\displaystyle x_n=\prod_{i=1}^{n-1}f_i(x_i),
$
we obtain immediately that this equation
reduces to
$
\displaystyle \log x_n=\sum_{i=1}^{n-1}\log f_i(x_i),
$
and this shows us that the hypersurface is separable. It is quite interesting that the family of separable hypersurfaces includes not only translation and homothetical hypersurfaces as particular subfamilies, but also some more general sets of graph hypersurfaces, namely quasi-sum and  quasi-product hypersurfaces, which are of particular interest in production theory \cite{Du2022,FW17,Luo2023}. We would like to point out
that the classification of the separable surfaces with constant Gaussian
curvature is done in \cite{hl} (see \cite{Gronwal} for a particular case), while the classification of the separable hypersurfaces with zero
Gauss–Kronecker curvature is done in \cite{cww}. On the other hand, the classification of the separable surfaces with non-zero constant
mean curvature is done in \cite{hl2}, while in Lorentz-Minkowski space, separable minimal surfaces were classified
in \cite{Kaya2022}. Moreover, other classification results for the quasi-sum and quasi-product production models via the main curvature invariants of the corresponding hypersurfaces were obtained in \cite{acdv,c,Du2022,FW17,Luo2023}.

\quad Motivated by the previously mentioned articles, we investigate the problem of finding the separable hypersurfaces of constant sectional curvature in the Euclidean $n$-space. In the case $n=3$, since the problem is equivalent to find such surfaces of constant Gaussian curvature and was completely solved in \cite{hl}, we are interested in the dimension $n>3$. In Sect. 3 we first classify the separable hypersurfaces of null sectional curvature, obtaining that there are three types of such hypersurfaces, namely: the hyperplanes, a particular type of Cobb-Douglas hypersurface, and the product 
$\Gamma \times \r^{n-2}$, where $\Gamma$ is a curve with non-null curvature included in a coordinate $2$-plane of $\r^n$.
The case of non-zero constant sectional curvature is studied in Sect. 4. We prove that hyperspheres are the only separable hypersurfaces of nonzero constant sectional curvature in the Euclidean $n$-space, provided that $n>3$.

\section{Preliminaries}

\quad In this section we summarize the differential-geometrical properties of the hypersurfaces in the Euclidean ambient space, cf. \cite{c1,d}.

\quad Let $(\r^n , \langle \cdot , \cdot \rangle)$ be the Euclidean $n-$space and $\tilde{\nabla}$ the Levi-Civita connection on $\r^n$. We denote by $\mathbb{S}^{n-1}=\{ \mathbf{x} \in \r^n : \langle \mathbf{x} ,\mathbf{x} \rangle =1 \}$ the unit hypersphere of $\r^n$. Let $M^{n-1}$ be an orientable hypersurface of $\r^n$ and denote by $\nu$ the {\it Gauss map} of $M^{n-1}$, i.e. $\nu : M^{n-1} \to \mathbb{S}^{n-1}$ such that $\nu(p)$ is a unit normal vector field $N (p)$ on $M^{n-1}$ at $p \in M^{n-1}$.

\quad Let $T_pM^{n-1}$ be the tangent space of $M^{n-1}$ at $p \in M^{n-1}$. Then, the differential $d\nu$ is called the {\it shape operator} of $M^{n-1}$ where $d\nu_p$ is an endomorphism on $T_pM^{n-1}$, i.e. $d\nu_p : T_pM^{n-1} \to T_pM^{n-1}$ is a linear map. An important intrinsic invariant called {\it Gauss-Kronecker curvature} at $p \in M^{n-1}$ is defined as $\det(d\nu_p)$.

\quad We define the {\it second fundamental form} $\textrm{II}$ of $M^{n-1}$ as a symmetric and bilinear map given by
$$
\textrm{II}(X_p,Y_p)=\langle -d\nu(X_p),Y_p \rangle, \quad X_p,Y_p \in T_pM^{n-1}.
$$

Setting $h(X_p,Y_p)=\textrm{II}(X_p,Y_p)N (p)$, the {\it formula of Gauss} is now
$$
\tilde{\nabla}_XY=\nabla_XY+h(X,Y),
$$
where $\nabla$ is the induced Levi-Civita connection on $M^{n-1}$. Let $\Pi$ be a plane section of $T_pM^{n-1}$ spanned by an orthonormal basis $\{X_p,Y_p \} $. Then, we define the {\it sectional curvature} of $\Pi$ as \[K(X_p,X_p)=\langle R(X,Y)Y,X \rangle(p),\] where $R$ is the Riemannian curvature tensor of $M^{n-1}$ given by
$$
 R(X,Y)Z = \nabla_X\nabla_YZ-\nabla_Y\nabla_XZ-\nabla_{[X,Y]} Z,
$$
for the smooth vector fields $X,Y,Z$ tangent to $M^{n-1}$.

\quad We call $M^{n-1}$ a {\it flat hypersurface} if $R$ is identically $0$. It is clear from the definition of sectional curvature $K$ that $R$ determines $K$. On the other hand, it is known that $K$ determines $R$ (see \cite[p. 78]{o}).
Actually, if $K(\Pi)=0$ for every plane section $\Pi$ of $T_pM^{n-1}$, then it follows that  $R(X_p,Y_p)Z_p=0$ for every $X_p,X_p,Z_p \in T_pM^{n-1}$ (see \cite[Proposition 41]{o}).

\quad By the {\it equation of Gauss}, we have
$$
K(X_p,X_p)=\textrm{II}(X_p,X_p)\textrm{II}(Y_p,Y_p)-\textrm{II}(X_p,Y_p)^2.
$$

\quad Assume now that $M^{n-1}$ is a hypersurface given in implicit form. Explicitly, if $F(x_1,...,x_n)$ is a smooth real valued function on $\r^n$ and if $\grad F$  denotes the gradient of $F$ in $\r^n$ then we have
$$
M^{n-1}=\{ (x_1,...,x_n) \in \r^n: F(x_1,...,x_n)=0,\grad F \neq 0 \} .
$$%
The unit normal vector field is
$$
N =\frac{\grad F}{\| \grad F \| },
$$
where $\| \grad F \|$ is the Euclidean norm of $\grad F$. In addition, the sectional curvature is
$$
K(X,Y)=\frac{1}{\| \grad F \|^2}\left(H^F(X,X) H^F(Y,Y)-(H^F(X,Y))^2 \right ),
$$
where $H^F$ is the {\it Hessian} of $F$ in $M^{n-1}$ defined by
$$
H^F(X,Y) = \langle \nabla_X \nabla_F, Y \rangle .
$$

\quad  At the end of this section, we recall the useful result from \cite[Lemma 1]{cww} (see also \cite[Lemma 1]{hl}).
\begin{lemma}\cite{cww,hl} \label{pre-lemma}
Let $Q(u_1,...,u_n)$ be a smooth function in a domain $\Omega \subset \r^n$ and $\Pi$ a hyperplane of the form $u_1+...+u_n=0$. If $Q=0$ on the intersection $\Omega \cap  \Pi$, then
$$
\frac{\partial Q}{\partial u_1}=...=\frac{\partial Q}{\partial u_n}.
$$
\end{lemma}

\section{Flat separable hypersurfaces }
\quad In this section, we classify the flat separable hypersurfaces $M^{n-1}$ of $\r^n$.
But in view of the fact that $M^{n-1}$ is flat if and only
if the sectional curvature function $K$ is identically zero, it follows that the problem we want to study is equivalent to finding all separable hypersurfaces of null sectional curvature.

\quad To get the classification, we will first deduce a statement related to the sectional curvature of $M^{n-1}$. Let $(x_1,...,x_n)$ be the canonical coordinates of $\r^n$ and $\{\partial /\partial x_1 ,..., \partial /\partial x_n\}$ the coordinate vector fields. Consider a separable hypersurface $M^{n-1}$ defined by the implicit equation
\begin{equation}
f_1(x_1)+...+f_n(x_n)=0, \quad x_k \in I_k \subset \r, \quad k =1,...,n. \label{sep-1}
\end{equation}
Denote by $f'_k=df/dx_k$, $f''_k=d^2f/dx_k^2$ and so on. Then, the unit normal vector field is
\[
N=\frac{1}{\sqrt{\sum_{k=1}^n f_k^{'2}}}(f'_1,...,f'_n).
\]

\quad By the regularity, we may assume that at least one of $f_1,...,f_n$ is not constant. Without lose of generality, we may assume $f'_n(x_n) \neq 0$, for each $x_n \in I_n$. We obey this assumption through the paper. Then, a basis of the tangent space $T_pM^{n-1}$ at some $p \in M^{n-1}$ is $\{X_1(p) ,...,X_{n-1} (p)\}$, where
\[
X_i=\frac{\partial}{\partial x_i}-\left ( \frac{f'_i}{f'_n} \right)\frac{\partial}{\partial x_n}, \quad i=1,...,n-1.
\]

\quad The covariant differentiation of $N$ is
$$
\tilde{\nabla}_{X_i}N=\frac{1}{\sqrt{\sum_{k=1}^n f_k^{'2}}}\left( (0,...,f''_i,...0)-\frac{f'_i}{f'_n}  (0,0,...,f''_n) \right) + \text{normal component},
$$
and so, for every $ i,j \in\{1,...,n-1  \}, i \neq j$,
$$
\langle \tilde{\nabla}_{X_i}N , X_i  \rangle =\frac{f''_i +f_i^{'2}f''_n/f_n^{'2}}{\sqrt{\sum_{k=1}^n f_k^{'2}}}, \quad \langle \tilde{\nabla}_{X_i}N , X_j \rangle =\frac{f'_if'_jf''_n/f_n^{'2}}{\sqrt{\sum_{k=1}^n f_k^{'2}}}.
$$
In addition, for every $ i,j \in\{1,...,n-1  \}, i \neq j$,
$$
\langle X_i , X_i  \rangle = 1+\left ( \frac{f'_i}{f'_n} \right)^2, \quad \langle X_i , X_j  \rangle =\frac{f'_if'_j}{f_n^{'2}}.
$$
\quad Now, let $K(X_i,X_j)$ be the curvature of the plane section spanned by $\{X_i,X_j\}$, for $i,j \in \{1,...,n-1 \}$ and $i < j$. A direct calculation yields
\begin{equation}
K(X_i,X_j)=\frac{f_i^{' 2}f''_jf''_n+f_j^{' 2}f''_if''_n+f_n^{'2}f''_if''_j}{\left ( \sum_{k=1}^nf_i^{'2} \right )(f_i^{'2}+f_j^{'2}+f_n^{'2})}, \quad i,j \in \{1,...,n-1 \},\quad i < j.  \label{sect.cur.}
\end{equation}

\quad The following result completely classifies the separable hypersurfaces $M^{n-1}$ of null sectional curvature under the condition $n>3$.

\begin{theorem} \label{th-1}
A separable hypersurface $M^{n-1}$ in $\r^n$ $(n>3)$ of null sectional curvature  is congruent to one of the following three hypersurfaces:
\begin{enumerate}
\item[(i)] a hyperplane,
\item[(ii)] $\Gamma \times \r^{n-2}$, where $\Gamma$ is a curve with non-null curvature included in a coordinate $2-$plane of $\r^n$,
\item[(iii)] $x_n=A\sqrt{ x_1...x_{n-1} }$, where $A$ is some positive constant.
\end{enumerate}
\end{theorem}

\begin{proof}
By the assumption, we have
\begin{equation}
f_i^{' 2}f''_jf''_n+f_j^{' 2}f''_if''_n+f_n^{'2}f''_if''_j=0, \quad \text{for every } i,j \in \{1,...,n-1 \}, i < j.  \label{K=0}
\end{equation}
A trivial solution to Equation \eqref{K=0}  is obtained when each of  $f_1,...,f_n$ is an affine function. Obviously, this implies that $M^{n-1}$ is a hyperplane and we have the item (i) of Theorem \ref{th-1}. In addition, as can be seen in Equation \eqref{K=0}, independently from the values of $i$ or $j$, we have that $f_n^{'2}$ and $f''_n$ appear in each equation. So, we separate the investigation into two cases:

\bigskip

\textbf{Case 1.} Assume that $f_n$ is an affine function, say $f_n(x_n)=\lambda_n x_n+\mu_n$, $\lambda_n , \mu_n  \in \r$. Here, due to the assumption above, we have $\lambda_n \neq 0$. Equation \eqref{K=0} is now
$$
\lambda_n^2f''_if''_j=0, \quad \text{for every } i,j \in \{1,...,n-1 \},i < j,
$$
yielding that $f_i$ or $f_j$ is an affine function. Without lose of generality, we may take $f_j(x_j)=\lambda_j x_j+\mu_j$, $\lambda_j , \mu_j  \in \r$. Since $\lambda_n \neq 0$, we write
$$
x_n=\frac{-1}{\lambda_n}(f_1(x_1)+\lambda_2 x_2+...+\lambda_{n-1}x_{n-1}+\alpha), \quad \alpha=\mu_2+...+\mu_n ,
$$
which is indeed a cylindrical hypersurface with parametrization
\begin{equation*}
\left.
\begin{array}{l}
(x_1,...,x_{n-1})\mapsto \mathbf{x}(x_1,...,x_{n-1})\\
=(x_1,...,\frac{-1}{\lambda_n}(f_1(x_1)+\alpha)) +x_2(0,1,...,\frac{-\lambda_2}{\lambda_n})+...+x_{n-1}(0,...,1,\frac{-\lambda_{n-1}}{\lambda_n}).
\end{array}%
\right.
\end{equation*}
So, we have proved the item (ii) of Theorem \ref{th-1}. In addition, according to the result of Seo (see \cite[Theorem 1.2]{s}), this is also a translation hypersurface with null Gauss-Kronecker curvature.

\bigskip

\textbf{Case 2.} Assume that $f''_n \neq 0$ for every $x_n \in I_n$. It is important to point out that the roles of $f_i$ and $f_j$ are symmetric, that is, if one case is valid for $f_i$ then so is $f_j$. Hence, it is sufficient to discuss the cases relating to $f_i$.

\quad If $f_i$ is an affine function, then again we arrive to the item (ii) of Theorem \ref{th-1}. Suppose now that $f''_i \neq 0$ for every $x_i \in I_i$. By the symmetry, the same presume is also valid for $f_j$. Therefore, in the rest of Case 2, we will assume that
$$
\prod_{k=1}^nf''_k\neq 0, \quad \text{for every } x_k \in I_k.
$$

\quad From now on, we will use an analogous of the notations and arguments given in \cite{cww,hl} and \cite[p. 71]{n}. For this, set $u_k=f_k(x_k)$ $(k=1,...,n)$ such that $ u_1+...+u_n=0$. Next, we introduce
$$
X_k(u_k)=f^{'}_k(x_k)^2,  \quad k=1,...,n
$$
or, by the chain formula,
$$
X'_k(u_k)=2f''_k(x_k),  \quad k=1,...,n.
$$

\quad With these new notations, Equation \eqref{K=0} is now
\begin{equation}
X_iX'_jX'_n+X'_iX_jX'_n+X'_iX'_jX_n=0 ,  \quad
\text{for every } i,j \in \{1,...,n-1 \},\quad i < j, \label{K=0-new}
\end{equation}
and for every $(u_1,...,u_n)$ satisfying $ u_1+...+u_n=0$. We may rewrite \eqref{K=0-new} as
\begin{equation}
\frac{X_i}{X'_i}+\frac{X_j}{X'_j}+\frac{X_n}{X'_n}=0 , \quad \text{for every } i,j \in \{1,...,n-1 \},\quad i < j. \label{K=0-new2}
\end{equation}
Considering Lemma \ref{pre-lemma} and then differentiating Equation \eqref{K=0-new2} with respect to $u_i,u_j,u_n$, we may deduce
\begin{equation}
\left (\frac{X_i}{X'_i} \right)'=\left (\frac{X_j}{X'_j} \right)'=\left (\frac{X_n}{X'_n} \right)'= \alpha,  \quad \text{for every } i,j \in \{1,...,n-1 \}, i < j,\quad \alpha \in \r. \label{K=0-new3}
\end{equation}

\quad We have to distinguish two subcases:

\bigskip

\textbf{Subcase 2.1.} $\alpha =0$. This implies the existence of the nonzero constants $\lambda_1 ,.., \lambda_n$ such that
\[
X'_k=\frac{2}{\lambda_i} X_k,  \quad \text{for every } k=1,...,n.
\]
In terms of the previous notations, we have
\begin{equation}
f''_k = \frac{1}{\lambda_i} f_k^{'2},  \quad \text{for every } k=1,...,n, \label{lambda-second.deriv.}
\end{equation}
where, due to Equations \eqref{K=0} or \eqref{K=0-new2},
\begin{equation}
\lambda_i + \lambda_j + \lambda_n =0 ,
 \quad \text{for every } i,j \in \{1,...,n-1 \}, i < j. \label{lambda-2}
\end{equation}
By Equation \eqref{lambda-2}, we have the following system:
\begin{equation}
\left.
\begin{array}{l}
\lambda _{1}+\lambda _{2}+\lambda _{n}=0, \\
\lambda _{1}+\lambda _{3}+\lambda _{n}=0, \\
... \\
\lambda _{n-1}+\lambda _{n-2}+\lambda _{n}=0.%
\end{array}%
\right. \label{lambda-system}
\end{equation}
Simplifying in terms of $\lambda_n$,
$$
(n-2)(\lambda_1 +... + \lambda_{n-1} )+\frac{(n-1)(n-2)}{2}\lambda_n=0 ,
$$
or equivalently,
\begin{equation}
\lambda_n =\frac{-2}{n-1}(\lambda_1 +... + \lambda_{n-1}). \label{lambda-value}
\end{equation}
By the system \eqref{lambda-system} we may conclude $\lambda_i = \lambda_j$, for every $i,j \in \{1,...,n-1 \}$ and $i<j$. Set $\lambda_i =\lambda$ for every $i = 1,...,n-1 $. Hence, Equation \eqref{lambda-value} implies $\lambda_n=-2\lambda$.

\quad The solutions to Equation \eqref{lambda-second.deriv.} are
\begin{equation*}
\left.
\begin{array}{l}
f_i(x_i)=-\lambda \log (x_i+\mu_i)+\beta_i, \quad \mu_i , \beta_i \in \r,\\
f_n(x_n)=2\lambda \log (x_n+\mu_n)+\beta_n, \quad \mu_n , \beta_n \in \r,%
\end{array}%
\right.
\end{equation*}
for every $i=1,...,n-1$. Hence, Equation \eqref{sep-1} is now
\begin{equation*}
x_n+\mu_n=A\sqrt{ (x_1+\mu_1  )...(x_{n-1}+\mu_{n-1})} ,
\end{equation*}
where
\[
A=e^{-(\beta_1 +...+ \beta_n)/2\lambda}.
\]
Up to suitable translations of $x_1,...,x_n$, the statement of the item (iii) of Theorem \ref{th-1}  is proved.

\bigskip

\textbf{Subcase 2.2.} $\alpha \neq 0$. So, we can take $1/(2\alpha)$ in Equation \eqref{K=0-new3} instead of $\alpha$. The integration in Equation \eqref{K=0-new3} gives
\begin{equation}
\frac{X_i}{X'_i}=\frac{u_i+\mu_i}{2\alpha}, \quad i =1,...,n-1, \label{case2.2.-1}
\end{equation}
where $\mu_1,...,\mu_n \in \r$ and, due to Equation \eqref{K=0-new2},
\[
u_i+u_j+u_n+\mu_i+\mu_j+\mu_n=0, \quad \text{for every } i,j \in \{1,...,n-1 \}, i < j.
\]
Hence, we have
\[
\left.
\begin{array}{l}
u _{1}+u _{2}+u _{n}+\mu_1+\mu_2+\mu_n=0, \\
u _{1}+u _{3}+u _{n}+\mu_1+\mu_3+\mu_n=0, \\
... \\
u _{n-1}+u _{n-2}+u _{n}+\mu_{n-1}+\mu_{n-2}+\mu_n=0.%
\end{array}%
\right.
\]
Summing the above relations, we get
\[
(n-2)(u_1+...+u_{n-1}+\mu_1+...+\mu_{n-1})+\frac{(n-2)(n-1)}{2}(u_n+\mu_n)=0.
\]
Since $u_1+...+u_n=0$, the above equation writes as
\begin{equation}
\mu_1+...+\mu_n+\frac{n-3}{2}(u_n+\mu_n)=0. \label{case2.2.-2}
\end{equation}
Now, taking into account that $u_n=f_n(x_n)$ is not constant, it follows that $n$ must be $3$, which is not our case.
\end{proof}

\begin{remark}
  We point out that the hypersurface appearing in Theorem \ref{th-1}, item (iii), is a particular type of Cobb-Douglas hypersurface. Recall that a hypersurface $M^{n-1}$ of $\r^n$ is said to be a \emph{Cobb-Douglas hypersurface} if $M^{n-1}$ is the graph of the function $F$ defined by  $ F(x_1,...,x_{n-1})=A\prod_{i=1}^{n-1}x_i^{\alpha_i}$,
where $A,\alpha_1,...,\alpha_{n-1}$ are positive constants (see, e.g., \cite{VilcuAML}). Hence it is clear that the hypersurface appearing in Theorem \ref{th-1}, item (iii), is noting but a Cobb-Douglas hypersurface with $\alpha_1=...=\alpha_{n-1}=\frac{1}{2}$.
\end{remark}

\begin{remark}
From Theorem \ref{th-1} we can also obtain, by particularizing the type of separable hypersurface, the classifications of translation, factorable, quasi-sum and quasi-product hypersurfaces in $\r^n$ with vanishing sectional curvature, recovering in particular some known results  both in differential geometry and microeconomics. For example, if $M^{n-1}$ is a quasi-product hypersurface, then Theorem \ref{th-1} reduces to \cite[Theorem 3.3]{FW17}, while  if $M^{n-1}$ is a factorable hypersurface, then
Theorem \ref{th-1} reduces to \cite[Corollary 4.2 (viii)]{acdv}. We also note that if $M^{n-1}$ is a quasi-sum production hypersurface, then Theorem \ref{th-1} leads to a generalization of the classification established in \cite[Theorem 1.1 (iv$_3$)]{VV2015} under an additional hypothesis of interest in production theory, namely the so-called \emph{proportional marginal rate of substitution} property (for more details see \cite{VV2015}).
\end{remark}

\section{The case of non-zero constant sectional curvature}

\quad In this section, we will assume that $K(\Pi)$ is a nonzero constant, say $K(\Pi)=K_0/4$, $K_0 \neq 0$,  for every plane section $\Pi$ of $T_pM^{n-1}$ at some point $p \in M^{n-1}$. Hence, Equation \eqref{sect.cur.} writes as
\begin{equation}
\frac{K_0}{4}=\frac{f_i^{' 2}f''_jf''_n+f_j^{' 2}f''_if''_n+f_n^{'2}f''_if''_j}{\left ( \sum_{k=1}^nf_k^{'2} \right )(f_i^{'2}+f_j^{'2}+f_n^{'2})}, \quad \text{ for every } i,j \in \{1,...,n-1 \}, i < j.  \label{sect.cur.non}
\end{equation}
Obviously, none of $f'_1,...,f'_n$ is $0$ because otherwise we would have the contradiction $K_0 =0$. So, in terms of the notations that we have used in the previous section, we may rewrite Equation \eqref{sect.cur.non} as
\begin{equation}
K_0 (X_i+X_j+X_n) \sum_{k=1}^nX_k=X_iX'_jX'_n+X_jX'_iX'_n+X_nX'_iX'_j,   \label{sect.cur.non-new}
\end{equation}
for every $i,j \in \{1,...,n-1 \}, i < j$. Here we recall that $X_k \neq 0$ for every $k =1,...,n$.

\quad In order to give the main result of this section, we need to prove some lemmas which we will use later. But, as a prior investigation, we want to distinguish the situation $X'_k(u_k) =4\lambda$, for every $k =1,...,n$ and $\lambda \in \r$. Obviously, the case $\lambda=0$, namely $M^{n-1}$ is a hyperplane, is not our interest. Otherwise, i.e. $ \lambda \neq 0$, we have
$$
f_k(x_k)=\lambda(x_k + \mu_k)^2+\beta_k, \quad \text{for every } k =1,..., n,
$$
or, by Equation \eqref{sep-1},
$$
\sum_{k=1}^n(x_k + \mu_k)^2=-\frac{1}{\lambda}\sum_{k=1}^n \beta_k.
$$
Since $M^{n-1}$ is a hypersurface, the right-hand side is a positive real number and so it is a hypersphere, which is known to be a hypersurface of constant sectional curvature.

\quad From now on, we will discard  the case that $M^{n-1}$ is a hypersphere. Another investigation for the derivatives $X'_1,...,X'_n$ is as follows.

\begin{lemma} \label{lemma-K0-1}
None of $X'_1,...,X'_n$ is $0$ provided that $K_0$ is different from $0$.
\end{lemma}
\begin{proof}
The proof is by contradiction. Without lose of generality, we assume $X'_1=0$, or equivalently, $X_1 = \lambda_1$, $\lambda_1 \in \r$, $\lambda_1 \neq 0$. For $(i,j)=(1,2 )$, Equation \eqref{sect.cur.non-new} writes as
\begin{equation}
K_0 (\lambda_1+X_2+X_n) \left ( \lambda_1 +\sum_{ k=2}^nX_k \right )-\lambda_1X'_2X'_n=0.   \label{sect.cur.non-lemma}
\end{equation}
Considering Lemma \ref{pre-lemma} and then differentiating Equation \eqref{sect.cur.non-lemma} with respect to $u_2,...,u_n$, we get
\begin{equation}
\left.
\begin{array}{l}
K_0X'_2\left ( 2(\lambda_1 +X_2+X_n)+ \sum_{ k=3}^{n-1}X_k \right ) -\lambda_1X''_2X'_n
=K_0X'_3(\lambda_1 +X_2+X_n), \\
K_0X'_3(\lambda_1 +X_2+X_n)=K_0X'_{4}(\lambda_1 +X_2+X_n)\\
... \\
K_0X'_{n-2}(\lambda_1 +X_2+X_n)=K_0X'_{n-1}(\lambda_1 +X_2+X_n), \\
K_0X'_{n-1}(\lambda_1 +X_2+X_n)=K_0X'_n\left ( 2(\lambda_1 +X_2+X_n)+ \sum_{ k=3}^{n-1}X_k \right )\\ \ \ \ \ \ \ \ \ \ \ \ \ \ \ \ \ \ \ \ \ \ \ \ \ \ \ \ \ \ \  \ \ \ \ \  -\lambda_1X'_2X''_n,%
\end{array}
\right. \label{lemma-system}
\end{equation}
where, due to $\lambda_1 +X_2+X_n \neq 0$ for every $(u_2 , u_n)$, we may conclude
$$
X'_3=X'_4=...=X'_{n-1}.
$$
Performing an analogous argument, in the case $(i,j)=(1,3)$, to Equation \eqref{sect.cur.non-new} we conclude
$$
X'_2=X'_4=...=X'_{n-1},
$$
and so
$$
X'_2=X'_3=...=X'_{n-1}=\lambda, \quad \lambda \in \r, \lambda \neq 0.
$$
Now, from the first equality in Equation \eqref{lemma-system}, we derive that
\[
\lambda K_0\left (\lambda_1 + \sum_{ k=2}^{n}X_k \right )=0,
\]
where none of the terms can be $0$, a contradiction.
\end{proof}

\begin{lemma} \label{lemma-K0-2}
$X'_i-X'_j $ is different from $0$, for every $i,j \in \{1,...,n-1 \},$ $ i < j$.
\end{lemma}
\begin{proof}
On the contrary, assume that $X'_i = X'_j$ for some $i,j \in \{1,...,n-1 \},$ $ i < j$. Hence, it follows $X'_i = X'_j=\lambda$, for a nonzero constant $\lambda$. Equation \eqref{sect.cur.non-new} is now
\begin{equation}
K_0 (X_i+X_j+X_n) \sum_{k=1}^nX_k-\lambda^2 X_n-\lambda (X_i+X_j)X'_n =0, \label{Xij}
\end{equation}
for some $i,j \in \{1,...,n-1 \}$. Using Lemma \ref{pre-lemma} in Equation \eqref{Xij} and then differentiating with respect to $u_l$ and $u_m$ (with $l\neq m$), we get
\[
K_0 (X_i+X_j+X_n) (X'_l-X'_m)=0, \quad \text{ for every } l,m \in \{1,...,n-1 \}-\{i,j\},\quad l\neq m.
\]
Here, we may conclude
\[
X'_i = X'_j=\lambda, \quad X'_l=\mu, \text{ for every } l \in \{1,...,n-1 \}-\{i,j\},
\]
where $\mu$ is a nonzero constant. After applying Lemma \ref{pre-lemma} in Equation \eqref{Xij},  differentiating with respect to $u_i$ and $u_l$, we deduce
\[
\lambda K_0 \sum_{k=1}^n X_k+K_0 (\lambda - \mu)(X_i+X_j+X_n) -\lambda^2X_n'=0.
\]
By using the same argument in the last equation, we derive $\lambda = \mu$.

\quad On the other hand, we apply Lemma \ref{pre-lemma} in Equation \eqref{Xij}, differentiating with respect to $u_i$ and $u_n$, and we find
\[
\left.
\begin{array}{l}
\lambda K_0 \left (X_i+X_j+X_n + \sum_{k=1}^nX_k  \right )- \lambda^2 X'_n \\
=K_0 \left (X_i+X_j+X_n + \sum_{k=1}^nX_k  \right )X'_n - \lambda^2 X'_n -\lambda (X_i+X_j)X''_n,
\end{array}
\right.
\]
or equivalently,
\[
K_0 \left (X_i+X_j+X_n + \sum_{k=1}^nX_k  \right )(X'_n - \lambda)-\lambda (X_i+X_j)X''_n=0,
\]
Here, if $X''_n=0$, then it must be $X'_n = \lambda$. But this case has been already discarded, being the case of a hypersphere. So, we have $(X'_n - \lambda)X''_n \neq 0$, yielding
\begin{equation}
K_0 \left (X_i+X_j+X_n + \sum_{k=1}^nX_k  \right )-\lambda (X_i+X_j)\Omega=0, \label{omega}
\end{equation}
where $\Omega =X''_n/(X'_n - \lambda)$. Again, we apply Lemma \ref{pre-lemma} to this equation where we differentiate with respect to $u_i$ and $u_l$, obtaining $\Omega =K_0/\lambda$. Replacing in Equation \eqref{omega},
$$
X_n + \sum_{k=1}^nX_k=0.
$$
Here, by Lemma \ref{pre-lemma} we may conclude $X'_n=\lambda /2$ or $X''_n=0$, a contradiction.
\end{proof}

\quad Next, Equation \eqref{sect.cur.non-new} writes as
\begin{equation}
A_{ij}+B_{ij}X_n+C_{ij}X'_n=K_0X_n^2, \quad \text{ for every } i,j \in \{1,...,n-1 \}, i < j, \label{ABC}
\end{equation}
where
$$
\left.
\begin{array}{l}
A_{ij}(u_1,...,u_{n-1})=-K_0(X_i+X_j)\sum_{k=1}^{n-1}X_k, \\
B_{ij}(u_1,...,u_{n-1})=-K_0\left (X_i +X_j + \sum_{k=1}^{n-1}X_k \right ) +X'_iX'_j, \\
C_{ij}(u_i,u_j)=X_iX'_j+X'_iX_j.%
\end{array} \label{lemma-sys}
\right.
$$

\quad Our third lemma claims that nowhere the coefficient function $C_{ij}$ is $0$.

\begin{lemma} \label{lemma-K0-3}
The function $C_{ij}$ given in Equation \eqref{ABC} is always different from $0$.
\end{lemma}

\begin{proof}
By contradiction, suppose that
\[
X_iX'_j+X'_iX_j=0, \quad \text{ for some } i,j \in \{1,...,n-1 \},\quad i < j.
\]
This implies the existence of a nonzero constant $\lambda$ such that
\[
X'_i=\lambda X_i, \quad X'_j=-\lambda X_j, \quad \text{ for some } i,j \in \{1,...,n-1 \},\quad i < j.
\]
Hence, Equation \eqref{sect.cur.non-new} is
\begin{equation}
K_0 (X_i+X_j+X_n) \sum_{k=1}^nX_k+\lambda^2 X_iX_jX_n =0. \label{C=0}
\end{equation}
Applying Lemma \ref{pre-lemma} to this equation where we differentiate with respect to $u_i$ and $u_j$,
$$
\left.
\begin{array}{l}
K_0\lambda \left (X_i+X_j+X_n + \sum_{k=1}^nX_k \right)X_i +\lambda^3 X_iX_jX_n
\\
=-K_0\lambda \left (X_i+X_j+X_n + \sum_{k=1}^nX_k \right)X_j -\lambda^3 X_iX_jX_n,
\end{array}
\right.
$$
or equivalently,
\begin{equation}
K_0\left (X_i+X_j+X_n + \sum_{k=1}^nX_k \right)(X_i +X_j)+2\lambda^2 X_iX_jX_n=0, \label{C=0-1}
\end{equation}
for some $i,j \in \{1,...,n-1 \}$, $i < j$. From Equations \eqref{C=0} and \eqref{C=0-1},
\begin{equation}
\left (X_i+X_j+X_n + \sum_{k=1}^nX_k \right)(X_i +X_j)-2(X_i+X_j+X_n) \sum_{k=1}^nX_k=0. \label{C=0-2}
\end{equation}
By Lemma \ref{pre-lemma}, where we differentiate with respect to $u_i$ and $u_j$, we may conclude
$$
-X'_i \left (X_n+\sum_{i \neq k \neq j}^{n} X_k \right)  = -X'_j \left (X_n+\sum_{i \neq k \neq j}^{n} X_k \right)
$$
or equivalently,
$$
\left (X_n+\sum_{i \neq k \neq j}^{n} X_k \right) (X_i+X_j)=0, \quad \text{ for some } i,j \in \{1,...,n-1 \},\quad i < j.
$$
This is impossible in view of Lemma \ref{lemma-K0-2}.
\end{proof}

\quad We now go back to Equation \eqref{ABC}. We use Lemma \ref{pre-lemma} in Equation \eqref{ABC}, where we differentiate with respect to $u_i$ and $u_j$, obtaining
\begin{equation}
A_{ij,u_i}-A_{ij,u_j}+(B_{ij,u_i}-B_{ij,u_j})X_n+(C_{ij,u_i}-C_{ij,u_j})X'_n=0,   \label{ABCij}
\end{equation}
for every $i,j \in \{1,...,n-1 \}$, $i < j$. Here $A_{ij,u_i}=\partial A_{ij} / \partial u_i$ and so on. A direct calculation leads to
$$
\left.
\begin{array}{l}
D_{ij}:=A_{ij,u_i}-A_{ij,u_j}=-K_0(X'_i-X'_j)\left (X_i+X_j + \sum_{k=1}^{n-1}X_k \right ), \\
E_{ij}:=B_{ij,u_i}-B_{ij,u_j}=-2K_0(X'_i-X'_j)+X''_iX'_j-X'_iX''_j, \\
F_{ij}:=C_{ij,u_i}-C_{ij,u_j}=X''_iX_j-X_iX''_j.
\end{array}
\right.
$$
Hence, Equation \eqref{ABCij} is now
\begin{equation}
D_{ij}+E_{ij}X_n+F_{ij}X'_n=0, \quad \text{ for every } i,j \in \{1,...,n-1 \},\quad i < j. \label{DEFij}
\end{equation}

\quad As previously mentioned, we will present two lemmas more. Our next lemma claims that none of $D_{ij}$, $E_{ij}$ and $F_{ij}$ is identically $0$.

\begin{lemma} \label{lemma-K0-4}
The functions $D_{ij}$, $E_{ij}$ and $F_{ij}$ given in Equation \eqref{DEFij} are always different from $0$.
\end{lemma}

\begin{proof}
The proof is by contradiction and we do it into separated cases.

\textbf{Case 1.} $D_{ij}=0$, for some $i,j \in \{1,...,n-1 \}$, $i < j$. Then
\[
K_0(X'_i-X'_j)\left (X_i+X_j + \sum_{k=1}^{n-1}X_k \right )=0,
\]
which is possible only if $X'_i-X'_j=0$, for some $i,j \in \{1,...,n-1 \}$, $i < j$. However, this contradicts with Lemma \ref{lemma-K0-2}.

\textbf{Case 2.} $F_{ij}=0$, for some $i,j \in \{1,...,n-1 \}$, $i < j$. Then, $X''_iX_j-X_iX''_j =0$, yielding a constant $\lambda$ such that
$$
\frac{X''_i}{X_i}=\lambda =\frac{X''_j}{X_j}, \quad  \text{for some } i,j \in \{1,...,n-1 \}.
$$
If $\lambda =0$, then Equation \eqref{DEFij} implies
$$
-K_0(X'_i-X'_j)\left (X_i+X_j + \sum_{k=1}^{n-1}X_k \right )-2K_0(X'_i-X'_j)X_n=0,
$$
or equivalently,
$$
X_i+X_j +X_n+ \sum_{k=1}^{n}X_k=0,  \quad  \text{for some } i,j \in \{1,...,n-1 \}.
$$
Applying Lemma \ref{pre-lemma}, the contradiction $X'_i-X'_j=0$ can be derived. Assume now that $\lambda \neq 0$. Then, Equation \eqref{DEFij} writes as
\begin{equation}
\left.
\begin{array}{l}
G_{ij}:=-K_0(X'_i-X'_j)\left (X_i+X_j + \sum_{k=1}^{n-1}X_k \right ) \\
+ \left ( -2K_0(X'_i-X'_j)+\lambda (X_iX'_j-X'_iX_j) \right )X_n=0, \label{Gij}
\end{array}
\right.
\end{equation}
for some $i,j \in \{1,...,n-1 \}$. Using Lemma \ref{pre-lemma}, the partial derivatives of $G_{ij}$ with respect to $u_i$ and $u_j$ satisfy for some $i,j \in \{1,...,n-1 \}$:
$$
\left.
\begin{array}{l}
H_{ij}:= G_{ij,u_i}-G_{ij,u_j} \\
=-\lambda K_0(X_i+X_j)\left (X_i+X_j + \sum_{k=1}^{n-1}X_k \right ) \\ -2K_0(X'_i-X'_j)^2+ \left ( -2\lambda K_0 (X_i+X_j) + 2\lambda X'_iX'_j-2\lambda^2X_iX_j\right )X_n =0.
\end{array}
\right.
$$
Again proceeding the same argument,
\begin{equation}
\left.
\begin{array}{l}
 H_{ij,u_i}-H_{ij,u_j} =-\lambda K_0(X'_i-X'_j)\left (X_i+X_j + \sum_{k=1}^{n-1}X_k \right ) \\
-6 \lambda K_0 (X_i+X_j)(X'_i-X'_j)\\ +(-2\lambda K_0 (X'_i-X'_j) +4\lambda^2(X_iX'_j-X'_iX_j))X_n =0.
\end{array} \label{Hij}
\right.
\end{equation}
Multiplying Equation \eqref{Gij} by $-\lambda$ and adding Equation \eqref{Hij}, we get
\[
-2K_0(X_i+X_j)(X'_i-X'_j)+\lambda(X_iX'_j-X'_iX_j)X_n =0, \quad  \text{for some } i,j \in \{1,...,n-1 \}.
\]
Here the coefficient of $X_n$ is different from $0$ because otherwise we would have $X_i+X_j=0$ or $X'_i-X'_j=0$, this being impossible due to Lemma \ref{lemma-K0-2}. So, we conclude
$$
X_n =2K_0\frac{(X_i+X_j)(X'_i-X'_j)}{\lambda (X_iX'_j-X'_iX_j)}\quad  \text{for some } i,j \in \{1,...,n-1 \}.
$$
Substituting into Equation \eqref{Gij},
$$
 X_n +\sum_{i \neq k \neq j}^n X_k=0.
$$
Considering Lemmas \ref{pre-lemma} and \ref{lemma-K0-2}, the last equation yields a contradiction.

\textbf{Case 3.} $E_{ij}=0$, for some $i,j \in \{1,...,n-1 \}$, $i < j$. Then,
$$
2K_0(X'_i-X'_j)-X''_iX'_j+X'_iX''_j=0, \quad \text{ for some } i,j \in \{1,...,n-1 \},\quad i < j.
$$
Dividing by $X'_iX'_j$, one arrives to the existence of a constant $\lambda$ where
$$
\frac{X''_i+2K_0}{X'_i}= \lambda = \frac{X''_j+2K_0}{X'_j},  \quad \text{ for some } i,j \in \{1,...,n-1 \}, i < j,
$$
or equivalently
$$
X''_i= \lambda X'_i -2K_0,\quad X''_j= \lambda X'_j -2K_0, \quad \text{ for some } i,j \in \{1,...,n-1 \}.
$$
Replacing in Equation \eqref{DEFij},
\begin{equation}
\left.
\begin{array}{l}
L_{ij}:=-K_0(X'_i-X'_j)\left (X_i+X_j + \sum_{k=1}^{n-1}X_k \right ) \\
+(\lambda (X'_iX_j-X_iX'_j)+2K_0(X_i-X_j))X'_n=0. \label{Lij}
\end{array}
\right.
\end{equation}
Due to Lemma \ref{pre-lemma}, the partial derivatives of $L_{ij}$ satisfy for some $i,j \in \{1,...,n-1 \}$,
$$
\left.
\begin{array}{l}
L_{ij,u_i}-L_{ij,u_j}=-\lambda K_0(X'_i+X'_j) \left (X_i+X_j + \sum_{k=1}^{n-1}X_k \right )  \\
+4K^2_0 \left (X_i+X_j + \sum_{k=1}^{n-1}X_k \right ) -2K_0(X'_i-X'_j)^2 \\
+(\lambda^2 (X'_iX_j+X_iX'_j)-2\lambda K_0(X_i+X_j)-2\lambda X'_iX'_j+2K_0(X'_i+X'_j))X'_n =0.
\end{array}
\right.
$$
The same argument yields
\begin{equation}
\left.
\begin{array}{l}
-\lambda^2K_0(X'_i-X'_j)\left (X_i+X_j + \sum_{k=1}^{n-1}X_k \right ) \\ -6\lambda K_0 (X'_i+X'_j)(X'_i-X'_j) +24K^2_0(X'_i-X'_j)
\\
+(\lambda^3 (X'_iX_j-X_iX'_j)+2\lambda^2 K_0(X_i-X_j)-4\lambda K_0(X'_i-X'_j) )X'_n =0, \label{Eij}
\end{array}
\right.
\end{equation}
for some $i,j \in \{1,...,n-1 \}$. Here, due to Lemma \ref{lemma-K0-2}, $\lambda \neq 0$. Multiplying Equation \eqref{Lij} by $-\lambda^2$ and adding Equation \eqref{Eij}, one obtains
$$
M_{ij}:=-12K_0+3\lambda(X'_i+X'_j)+2\lambda X'_n=0, \quad \text{ for some } i,j \in \{1,...,n-1 \},\quad i < j,
$$
Using Lemma \ref{pre-lemma}, we have $M_{ij,u_i}-M_{ij,u_j}=0$, yielding the following contradiction
$$
3\lambda(X''_i-X''_j)=3\lambda^2(X'_i-X'_j)=0, \quad \text{ for some } i,j \in \{1,...,n-1 \},\quad i <j.
$$
\end{proof}

\quad Now, by Lemma \ref{lemma-K0-4}, Equation \eqref{DEFij} is ready so that Lemma \ref{pre-lemma} can be applied. Differentiating Equation \eqref{DEFij} with respect to $u_i$ and $u_j$,
\begin{equation}
D_{ij,u_i} - D_{ij,u_j} +(E_{ij,u_i} - E_{ij,u_j})X_n+(F_{ij,u_i} - F_{ij,u_j})X'_n=0,  \label{DEFijnew}
\end{equation}
for every $i,j \in \{1,...,n-1 \}$, $i < j$. The following is the last lemma before presenting the proof of the theorem.

\begin{lemma} \label{lemma-K0-5}
The functions $D_{ij}$, $E_{ij}$ and $F_{ij}$ satisfy
$$
\frac{D_{ij,u_i} - D_{ij,u_j}}{D_{ij}}=\frac{E_{ij,u_i} - E_{ij,u_j}}{E_{ij}}=\frac{F_{ij,u_i} - F_{ij,u_j}}{F_{ij}},
$$
or equivalently,
\begin{equation}
\left ( \frac{D_{ij}}{E_{ij}} \right )_{u_i}=\left ( \frac{D_{ij}}{E_{ij}} \right )_{u_j}, \text{ } \left ( \frac{D_{ij}}{F_{ij}}\right )_{u_i}=\left (\frac{D_{ij}}{F_{ij}}\right )_{u_j}, \text{ }\left ( \frac{E_{ij}}{F_{ij}}\right )_{u_i}=\left ( \frac{E_{ij}}{F_{ij}} \right )_{u_j}, \label{defratio}
\end{equation}
for every $i,j \in \{1,...,n-1 \}$, $i < j$.
\end{lemma}

\begin{proof}
Equation \eqref{ABC} is explicitly rewritten as
\begin{equation}
\left.
\begin{array}{l}
(X_iX'_j+X'_iX_j)X'_n=K_0(X_i+X_j)\sum_{k=1}^{n-1}X_k+ \\
\left ( K_0\left (X_i +X_j + \sum_{k=1}^{n-1}X_k \right ) -X'_iX'_j \right )X_n +K_0X_n^2 ,
\end{array} \label{lemma-K0-5-1}
\right.
\end{equation}
for every $i,j \in \{1,...,n-1 \}$, $i < j$. We will eliminate the term $X'_n$ from Equations \eqref{DEFij} and \eqref{lemma-K0-5-1}. We first multiply Equations \eqref{DEFij} and \eqref{lemma-K0-5-1} by $-(X_iX'_j+X'_iX_j)$ and by $F_{ij}$, respectively. We then add the resulting equalities, deducing
\begin{equation}
P_{ij}+Q_{ij}X_n+R_{ij}X_n^2=0, \quad \text{ for every } i,j \in \{1,...,n-1 \},\quad i < j, \label{lemma-K0-5-2}
\end{equation}
where
$$
\left.
\begin{array}{l}
P_{ij}=K_0F_{ij}(X_i+X_j)\sum_{k=1}^{n-1}X_k+ D_{ij}(X_iX'_j+X'_iX_j) \\
Q_{ij}= F_{ij} \left (K_0\left (X_i +X_j + \sum_{k=1}^{n-1}X_k \right ) -X'_iX'_j \right )+E_{ij}(X_iX'_j+X'_iX_j) \\
R_{ij}= K_0F_{ij},
\end{array}
\right.
$$
for every $i,j \in \{1,...,n-1 \}$, $i < j$. By Lemma \ref{pre-lemma}, Equation \eqref{lemma-K0-5-2} yields
\begin{equation}
P_{ij,u_i}-P_{ij,u_j}+(Q_{ij,u_i}-Q_{ij,u_j})X_n+(R_{ij,u_i}-R_{ij,u_j})X_n^2=0,  \label{lemma-K0-5-3}
\end{equation}
for every $i,j \in \{1,...,n-1 \}$, $i < j$. Here the coefficients are
$$
\left.
\begin{array}{l}
P_{ij,u_i}-P_{ij,u_j}=\\
K_0(F_{ij,u_i}-F_{ij,u_j})(X_i+X_j)\sum_{k=1}^{n-1}X_k   +(D_{ij,u_i}-D_{ij,u_j})(X_iX'_j+X_iX'_j),\\
Q_{ij,u_i}-Q_{ij,u_j}= \\
(F_{ij,u_i}-F_{ij,u_j})\left (K_0\left (X_i +X_j + \sum_{k=1}^{n-1}X_k \right ) -X'_iX'_j \right ) +(E_{ij,u_i}-E_{ij,u_j})(X_iX'_j+X'_iX_j), \\
R_{ij,u_i}-R_{ij,u_j}=K_0(F_{ij,u_i}-F_{ij,u_j}).
\end{array}
\right.
$$
Now we have two separated cases according as $R_{ij,u_i}-R_{ij,u_j}$ is or is not $0$.

\textbf{Case 1.} $R_{ij,u_i}-R_{ij,u_j} \neq 0$. 
This assumption is equivalent to $F_{ij,u_i}-F_{ij,u_j} \neq 0$. 
Notice that Equations \eqref{lemma-K0-5-2} and \eqref{lemma-K0-5-3} have the common pair of roots. Hence, we get
\[
\frac{P_{ij,u_i} - P_{ij,u_j}}{P_{ij}}=\frac{Q_{ij,u_i} - Q_{ij,u_j}}{Q_{ij}}=\frac{R_{ij,u_i} - R_{ij,u_j}}{R_{ij}}.
\]
Equivalently, we have
\begin{equation}
\left.
\begin{array}{l}
P_{ij}(R_{ij,u_i} - R_{ij,u_j})=R_{ij}(P_{ij,u_i} - P_{ij,u_j}) \\
Q_{ij}(R_{ij,u_i} - R_{ij,u_j})=R_{ij}(Q_{ij,u_i} - Q_{ij,u_j})
\end{array}
\right. \label{pqrij}
\end{equation}
and a direct calculation in Equation \eqref{pqrij} implies
\[
\left.
\begin{array}{l}
K_0D_{ij}(X_iX'_j+X'_iX_j)(F_{ij,u_i}-F_{ij,u_j})=K_0F_{ij}(X_iX'_j+X'_iX_j)(D_{ij,u_i}-D_{ij,u_j}) \\
K_0E_{ij}(X_iX'_j+X'_iX_j)(F_{ij,u_i}-F_{ij,u_j})=K_0F_{ij}(X_iX'_j+X'_iX_j)(E_{ij,u_i}-E_{ij,u_j}).
\end{array}
\right.
\]
In view of Lemmas \ref{lemma-K0-3} and \ref{lemma-K0-4}, as well as our assumption, none of the terms in the above equation is $0$, proving the result of the lemma.

\textbf{Case 2.} $R_{ij,u_i}-R_{ij,u_j} = 0$. 
 Equation \eqref{lemma-K0-5-3} is now
\[
P_{ij,u_i}-P_{ij,u_j}+(Q_{ij,u_i}-Q_{ij,u_j})X_n=0.
\]
Here $Q_{ij,u_i}-Q_{ij,u_j} = 0$ (or $\neq 0$) if and only if $P_{ij,u_i}-P_{ij,u_j} = 0$ (or $\neq 0$). 
 There is nothing to prove in the case $Q_{ij,u_i}-Q_{ij,u_j} = 0$. 
 Assume now that $Q_{ij,u_i}-Q_{ij,u_j} \neq 0$. 
  Then we have
\[
X_n=-\frac{P_{ij,u_i}-P_{ij,u_j}}{Q_{ij,u_i}-Q_{ij,u_j}}.
\]
Considering in Equation \eqref{lemma-K0-5-2}, we understand that the value of $X_n$ is unique. 
 Hence,
\[
X_n=-\frac{Q_{ij}}{2R_{ij}}
\]
and by Lemma \ref{pre-lemma} we get
\[
\left ( \frac{Q_{ij}}{R_{ij}} \right )_{u_i} =\left ( \frac{Q_{ij}}{R_{ij}} \right )_{u_j}. 
\]
Equivalently, we have
\[
\frac{(Q_{ij,u_i}-Q_{ij,u_j})R_{ij}-Q_{ij}(R_{ij,u_i}-R_{ij,u_j})}{R_{ij}^2}=0, 
\]
yielding $Q_{ij,u_i}-Q_{ij,u_j}=0$. 
This is a contradiction and completes the proof. So, there are two functions $S_{ij}, T_{ij}$ of single variable
\[
S_{ij}(u_1+...+u_{n-1})=\frac{D_{ij}(u_1,...,u_{n-1})}{F_{ij}(u_1,...,u_{n-1})}, \quad
T_{ij}(u_1+...+u_{n-1})=\frac{E_{ij}(u_1,...,u_{n-1})}{F_{ij}(u_1,...,u_{n-1})}.
\]
\end{proof}

\quad After all these preparatory results, we are ready to present the following result.

\begin{theorem} \label{th-2}
Hyperspheres are the only separable hypersurfaces in $\r^n$ $(n>3)$ having nonzero constant sectional curvature.
\end{theorem}
\begin{proof}
By contradiction, assume that $M^{n-1}$ is not a hypersphere but has constant sectional curvature $K_0/4$, $K_0 \neq 0$. By Lemma \ref{lemma-K0-4}, $R_{ij}= K_0F_{ij}$ is always different from $0$ for every $i,j \in \{1,...,n-1 \}$, $i<j$. Hence, Equation \eqref{lemma-K0-5-2} is
\begin{equation}
\frac{P_{ij}}{R_{ij}}+\frac{Q_{ij}}{R_{ij}}X_n+X_n^2=0, \quad \text{ for every } i,j \in \{1,...,n-1 \},\quad i < j \label{lemma-K0-th-pqr}.
\end{equation}
Denote by $X_{n,1}$ and $X_{n,2}$ the roots of this equation. So,
$$
X_{n,1}+X_{n,2}=-\frac{Q_{ij}}{R_{ij}}, \quad X_{n,1}X_{n,2}=\frac{P_{ij}}{R_{ij}},
$$
where the statements in the right-hand sides depend on the variables $u_1,...,u_{n-1}$, while those in the left-hand sides depend on the variable $u_n$. 
Now, we can set
$$
\tilde{S}_{ij}(u_1+...+u_{n-1})=\frac{P_{ij}(u_1,...,u_{n-1})}{R_{ij}(u_1,...,u_{n-1})}, \quad \tilde{T}_{ij}(u_1+...+u_{n-1})=\frac{Q_{ij}(u_1,...,u_{n-1})}{R_{ij}(u_1,...,u_{n-1})}.
$$
Obviously, we have $\tilde{T}_{ij,u_i}-\tilde{T}_{ij,u_j}=0$,  for every  $i,j \in \{1,...,n-1 \}$, $i < j. $
On the other hand, Equation \eqref{lemma-K0-th-pqr} becomes
$$
\tilde{S}_{ij}+\tilde{T}_{ij}X_n+X_n^2=0, \quad \text{ for every } i,j \in \{1,...,n-1 \},\quad i < j.
$$
By Lemma \ref{pre-lemma}, differentiating with respect to $u_i$ and $u_n$, we get
\begin{equation}
(2X_n+T_{ij})X'_n-T_{ij,u_i}X_n-S_{ij,u_i}=0, \quad \text{ for every } i,j \in \{1,...,n-1 \},\quad i < j. \label{XST}
\end{equation}
We will eliminate the term $X'_n$ from Equations \eqref{ABC} and \eqref{XST}. We first multiply Equations \eqref{ABC} and \eqref{XST} by $-(2X_n+\tilde{T}_{ij})$ and by $C_{ij}$, respectively. We then add the resulting equalities, obtaining a polynomial equation on $X_n$ of degree $3$
$$
-\frac{A\tilde{T}_{ij}+C_{ij}\tilde{S}_{ij,u_i}}{2K_0}-\frac{2A_{ij}+B_{ij}\tilde{T}_{ij}+C_{ij}\tilde{T}_{ij,u_i}}{2K_0}X_n + \frac{K_0\tilde{T}_{ij}-2B_{ij}}{2K_0}X^2_n +X^3_n=0,
$$
for every $i,j \in \{1,...,n-1 \}$, $i < j$. Here we denote the roots $X_{n,1}$, $X_{n,2}$ and $X_{n,3}$. Hence, the sum of the roots are the negative of the coefficient of $X^2_n$, namely
$$
X_{n,1}+X_{n,2}+X_{n,3}=\frac{B_{ij}}{K_0}-\frac{\tilde{T}_{ij}}{2}, \quad \text{ for every } i,j \in \{1,...,n-1 \},\quad i < j,
$$
where the statement in the right-hand side depend on the variables $u_1,...,u_{n-1}$, while the others depend on the variable $u_n$. Using Lemma \ref{pre-lemma}, where we differentiate with respect to $u_i$ and $u_j$, we obtain
$$
\frac{1}{K_0}(B_{ij,u_i}-B_{ij,u_j})-\frac{1}{2}(\tilde{T}_{ij,u_i}-\tilde{T}_{ij,u_j})=0, \quad \text{ for every } i,j \in \{1,...,n-1 \},\quad i < j,
$$
or equivalently,
$$
\frac{1}{K_0}(B_{ij,u_i}-B_{ij,u_j})=0, \quad \text{ for every } i,j \in \{1,...,n-1 \},\quad i < j.
$$
Since $B_{ij,u_i}-B_{ij,u_j}=E_{ij}$, by Lemma \ref{lemma-K0-4} we arrive to a contradiction.

\end{proof}

\subsection*{Acknowledgment}
Rafael López is partially supported by MINECO/MICINN/FEDER grant no. PID2020-117868GB-I00,  and by the ``Mar\'{\i}a de Maeztu'' Excellence Unit IMAG, reference CEX2020-001105- M, funded by MCINN/AEI/ 10.13039/501100011033/ CEX2020-001105-M. Gabriel-Eduard V\^{\i}lcu was supported by a grant of the Ministry of Research, Innovation and Digitization, CNCS/CCCDI - UEFISCDI, project number PN-III-P4-ID-PCE-2020-0025, within PNCDI III.


 \end{document}